\documentclass[11pt,fleqn]{article}
\usepackage{amsmath, amsthm, amssymb, amsfonts, verbatim,color,esint}
\usepackage[pdftex]{graphicx}
\usepackage[margin=3cm]{geometry}
\usepackage{csquotes}
\usepackage[bookmarks]{hyperref}

\newcommand{\R}{{\mathbb R}}
\newcommand{\C}{{\mathbb C}}
\newcommand{\Z}{{\mathbb Z}}
\newcommand{\e}{{\epsilon}}
\DeclareMathOperator{\dist}{dist}
\DeclareMathOperator{\dv}{div}
\DeclareMathOperator{\supp}{supp}
\DeclareMathOperator{\cof}{cof}

\DeclareMathOperator{\re}{Re}
\DeclareMathOperator{\im}{Im}

\newcommand{\lt}{\left}
\newcommand{\rt}{\right}

\newcommand{\nn}{\nonumber}

\newcommand{\qd}{\quad}

\newcommand{\wt}{\widetilde}

\newcommand{\na}{\nabla}

\newtheorem{theorem}{Theorem}[section]

\newtheorem{lemma}[theorem]{Lemma}
\newtheorem{proposition}[theorem]{Proposition}
\theoremstyle{definition}

\theoremstyle{remark}
\newtheorem{remark}[theorem]{Remark}

\title{Quantitative rigidity of differential inclusions in two dimensions}

\date{}
\author{Xavier Lamy\footnote{Institut de Math\'ematiques de Toulouse, UMR 5219, Universit\'e de Toulouse, CNRS, UPS
		IMT, F-31062 Toulouse Cedex 9, France. Email: Xavier.Lamy@math.univ-toulouse.fr}
	\and Andrew Lorent\footnote{Department of Mathematical Sciences, University of Cincinnati, Cincinnati, OH 45221, USA. Email: lorentaw@uc.edu} 
	\and Guanying Peng\footnote{Department of Mathematical Sciences, Worcester Polytechnic Institute, Worcester, MA 01609, USA. Email: gpeng@wpi.edu}}

\begin{document}

\maketitle

\begin{abstract}
For any compact connected one-dimensional submanifold $K\subset \mathbb R^{2\times 2}$ which has no rank-one connection and is elliptic, 
we prove the quantitative rigidity estimate
\begin{align*}
\inf_{M\in K}\int_{B_{1/2}}| Du -M |^2\,dx  \leq C \int_{B_1} \dist^2(Du, K)\, dx,
\qquad\forall u\in H^1(B_1;\mathbb R^2).
\end{align*}
This is an optimal generalization, for compact connected submanifolds of $\mathbb R^{2\times 2}$, of the celebrated quantitative rigidity estimate  
of Friesecke, James and M\"uller for the approximate differential inclusion into $SO(n)$.
The proof relies on
the special properties of elliptic subsets $K\subset\R^{2\times 2}$ with respect to conformal-anticonformal decomposition, which  provide
a quasilinear elliptic PDE satisfied by solutions of the exact differential inclusion $Du\in K$.
We also give an example showing that no analogous result can hold true in $\mathbb R^{n\times n}$ for $n\geq 3$. 

\end{abstract}

\section{Introduction}

In 1850, Liouville \cite{LIO50} proved that, given a domain $\Omega\subset\R^3$, 
any smooth map $u\colon\Omega\to \R^3$ satisfying the differential inclusion $Du(x)\in \R_+ O(n)$ for all $x\in\Omega$ must be either affine or a M\"obius transform.
A corollary to Liouville’s Theorem
is that a $C^3$
function whose gradient belongs everywhere to $SO(n)$ is an affine mapping. 
This phenomenon of being able to 
globally control a 
map
satisfying a certain differential inclusion $Du\in K$
 is known as ``rigidity''.

In \cite{FJM02} Friesecke, James and M\"{u}ller solved a long standing open problem by proving an optimal 
quantitative rigidity 
estimate for $K=SO(n)$.
Specifically, they showed that for every bounded open connected Lipschitz domain $U\subset \mathbb{R}^n$,
$n\geq 2$, 
there exists a constant $C(U)$ such that, for $K = SO(n)$,
\begin{equation}\label{e:FJM}
\inf_{R\in K}  \| Dv - R\|_{L^2(U)} \leq C (U) \| \dist(Dv, K) \|_{L^2(U)},\qquad\forall
v\in H^{1}\lt(U;\mathbb{R}^n\rt).
\end{equation}
Here and below, $\dist(M, K)$ denotes the distance from a matrix $M\in \mathbb{R}^{n\times n}$ to a subset 
$K\subset \mathbb{R}^{n\times n}$ measured in the Euclidean norm. This result strengthened earlier work of a series of authors,
including John \cite{J61}, Reshetnyak  \cite{RES67}, and Kohn \cite{K82}, and it has had a number
of important applications, 
 in particular to thin film limits of elastic structures 
 \cite{FJM02,FJM06}.
 One of the 
remarkable things about this result is that it is a striking fact about classical mathematical objects  that
 could be understood by the mathematicians of  
 hundreds of years ago. 

A number of works have extended the above result \eqref{e:FJM} to cover various larger classes of
matrices than $K=SO (n)$. 
Chaudhuri and M\"uller \cite{CM04} and later
De Lellis and Szekelyhidi \cite{DS09} considered a set of the form $K = SO (n) A\cup SO (n) B$
where $A$ and $B$ are strongly incompatible in the sense of Matos  \cite{M92}. 
Faraco and Zhong \cite{FZ05} proved  
 an analogous quantitative rigidity result with $K=\mathfrak m\cdot SO(n)$ 
where $\mathfrak m \subset (0,+\infty)$  is compact.
There the infimum in the left-hand side of \eqref{e:FJM} also needs to include the gradients of M\"obius transforms, and the integral is over a smaller subset $U'\subset\subset U$.

Our main result is an optimal generalization of the quantitative rigidity estimate of \cite{FJM02} 
in the context of compact connected submanifolds $K\subset \mathbb{R}^{2\times 2}$ without boundary.
 
\begin{theorem}\label{t:main}
Let $K\subset\R^{2\times 2}$ be a smooth, compact  and connected  1-manifold without rank-one connections, and elliptic: there exists $C_*>0$ such that
\begin{align}\label{eq:ellipt}
|M-M'|^2\leq  C_* \det(M-M')\qquad\forall M,M'\in K.
\end{align}
 Then for any $u\in H^1(B_1;\R^2)$ we have
\begin{align}\label{eq:mainest}
\inf_{M\in K}\int_{B_{1/2}}| Du -M |^2\,dx  \leq C \int_{B_1} \dist^2(Du, K)\, dx,
\end{align}
for some constant $C=C(K)>0$.
\end{theorem}

\begin{remark}
A covering argument as in \cite{FZ05} shows that the estimate \eqref{eq:mainest} in the  balls $B_{1/2}\subset B_1$ automatically improves to 
\begin{align*}
\inf_{M\in K}\int_{\Omega'}| Du -M |^2\,dx  \leq C \int_{\Omega} \dist^2(Du, K)\, dx,\qquad\forall u\in H^1(\Omega;\R^2),
\end{align*}
for open bounded sets $\Omega'\subset\subset\Omega\subset\R^2$ and a constant $C=C(K,\Omega',\Omega)>0$.
\end{remark}

This result is optimal among compact connected submanifolds $K\subset\mathbb R^{2\times 2}$ without boundary
for the following reasons:
\begin{itemize}
\item First, it is classical that the no-rank-one-connections assumption is necessary for the rigidity of the exact differential inclusion (see e.g. \cite{K03}).
\item Second, ellipticity is necessary for the validity of the linearized version of \eqref{eq:mainest} (see Remark~\ref{r:lin}), because non-ellipticity would imply that the tangent space $T_{M}K$ has a rank-one connection for some $M\in K$. 
\item
Third, the two previous conditions (no rank-one connections and ellipticity) imply that the submanifold $K$ must be of dimension 1 \cite[Corollary~3.5 \& 3.6]{zhang97}.
\end{itemize}
Moreover, we provide in Section~\ref{s:3x3} an example showing that the two-dimensional setting is also optimal:
there exists an elliptic 1-submanifold $K\subset\R^{3\times 3}$ without rank-one connection but which contains a so-called $\mathcal T_4$ configuration, a well-known obstruction to compactness of sequences $\lbrace u_k\rbrace\subset H^1$ satisfying $\dist(Du_k,K)\to 0$ in $L^2$ \cite{BFJK94}, and therefore to any type of quantitative rigidity estimate.

One of our motivations for studying differential inclusion into general submanifolds $K\subset\R^{2\times 2}$ is our previous work \cite{LLP20} where we obtained a rigidity result for a non-elliptic differential inclusion related to the so-called Aviles-Giga functional, 
and pointed out the nice consequences that a corresponding quantitative rigidity estimate would have.
Theorem~\ref{t:main} is not valid for non-elliptic differential inclusions, 
but the ideas in the present work should be relevant to attain that goal.

While the statements of the quantitative rigidity results of \cite{FJM02,CM04,FZ05}
   are elementary, 
their proofs are not. 
Their starting point, in addition to rigidity of the exact differential inclusion,
is a linearized version of \eqref{e:FJM} for the differential inclusion $Du\in T_{M_0}K$ into a tangent space $T_{M_0}K$.
For $K=SO(n)$ and $M_0=I$, this is Korn's inequality.
A natural linearization procedure then provides a quantitative rigidity estimate, 
but in terms of the $L^\infty$ norm of $\dist(Du,K)$, rather than $L^2$.
Strengthening the $L^2$ bound on $\dist(Du,K)$ into an $L^\infty$ bound constitutes therefore the main difficulty.
A key idea, introduced in \cite{FJM02},
is to use the regularity of an elliptic PDE satisfied by solutions of the exact differential inclusion:
for $K=SO(n)$ the exact differential 
inclusion $Du\in SO\lt(n\rt)$ implies that the coordinate functions $u_k$ are harmonic. For $K\subset \mathbb{R}_{+}SO\lt(n\rt)$ the coordinate functions satisfy the $(n-2)$-Laplace equation $\dv (\lt|\na u_k\rt|^{n-2} \na u_k )=0$.
Such PDE follows from the universal identity $\dv\cof(Du)=0$ (where $\cof$ denotes the cofactor matrix), together with identities satisfied by matrices in the specific set $K$.
It is satisfied by solutions of the exact differential inclusion, 
and for a general map $u$ the error from solving that PDE can be controlled in terms of the right-hand side of \eqref{e:FJM}. 
This allows to reduce the proof of \eqref{e:FJM} to maps solving that PDE. Elliptic regularity  then provides, via a compactness argument, a uniform bound on $\dist(Du,K)$ and the linearization can be performed.

Following this scheme, 
the main ingredient to prove Theorem~\ref{t:main} is to 
embed $K$ into the graph of a uniformly monotone vector field: this will be enough to turn the identity $\dv\cof(Du)=0$ into a quasilinear elliptic equation for the exact differential inclusion $Du\in K$.

\begin{proposition}\label{p:G}
Let $K$ be as in Theorem~\ref{t:main}.
There exist $G_1,G_2\colon\R^2\to\R^2$ smooth, globally Lipschitz vector fields such that
\begin{align*}
K \subset 
\left\lbrace 
\left(\begin{array}{c} A \\ iG_1(A) \end{array}\right)\colon A\in\R^2
\right\rbrace
\cap
\left\lbrace 
\left(\begin{array}{c} -iG_2(B) \\ B \end{array}\right)\colon B\in\R^2
\right\rbrace,
\end{align*}
and $G_1,G_2$ are uniformly monotone, that is
\begin{align*}
(G_j(X)-G_j(X'))\cdot (X-X') \geq \lambda |X-X'|^2\qquad\forall X,X'\in\R^2,
\end{align*}
for some constant $\lambda >0$ depending only on $K$.
\end{proposition}

Proposition~\ref{p:G} relies on 
remarkable properties of elliptic
subsets of $\R^{2\times 2}$ with respect to the decomposition into 
conformal and anticonformal parts, 
discovered in \cite{zhang97} and  exploited in a striking manner in \cite{FS08} (see also \cite{K03}).
(It is also related to the classical link between two-dimensional elliptic PDEs of second order and complex Beltrami equations, see e.g. the introduction of \cite{AIM09} and references therein.)
Then the proof of Theorem~\ref{t:main} follows the scheme outlined above.

The article is organized as follows. 
In Section~\ref{s:basic} we establish the two basic prerequisites to Theorem~1, rigidity for the exact differential inclusion and the linearized estimate. 
In Section~\ref{s:G} we give the proof of  Proposition~\ref{p:G}.
In Section~\ref{s:proof} we gather these ingredients to prove Theorem~\ref{t:main}.
In Section~\ref{s:3x3} we describe the counterexample  in $\R^{3\times 3}$.

\medskip

\noindent\bf  Acknowledgments. \rm  X. L. received support from ANR project ANR-18-CE40-0023. A. L. gratefully acknowledges the support of the Simons foundation, collaboration grant \#426900. G. P. was supported in part by NSF grant DMS-2206291.

\section{Basic ingredients: rigidity of the exact inclusion and linearized estimate}\label{s:basic}

In this section we prove the following two Lemmas.

\begin{lemma}\label{l:exact}
If $u\in H^1(B_1;\R^2)$ is such that $Du\in K$ a.e., then $Du\equiv M$ for some $M\in K$. 
\end{lemma}

\begin{lemma}\label{l:lin}
For all $M\in K$ and $u\in H^1(B_1;\R^2)$ we have
\begin{align}\label{eq:lin}
\inf_{X\in T_M K} \int_{ B_{1}}|Du-X|^2\, dx \leq C \int_{B_1} \dist^2(Du,T_M K)\, dx
\end{align}
for some constant $C=C(K)>0$, where $T_M K$ denotes the linear tangent space to $K$ at $M$.
\end{lemma}

\begin{remark}\label{r:lin}
The linearized estimate \eqref{eq:lin}, or rather its weaker interior version
\begin{align}\label{eq:linweak}
\inf_{X\in T_M K} \int_{ B_{1/2}}|Du-X|^2\, dx \leq C \int_{B_1} \dist^2(Du,T_M K)\, dx,
\end{align}
is a necessary condition for \eqref{eq:mainest} to be valid. Assume indeed that \eqref{eq:mainest} is verified, fix $u\in C^1(\overline B_1;\R^2)$, and apply \eqref{eq:mainest} to $v_\e(x)=Mx +\e u(x)$ for $\e\ll 1$. There exists $M_\e\in K$ such that
\begin{align*}
\int_{B_{1/2}}|M-M_\e -\e Du|^2\, dx &\leq C \int_{B_1}\dist^2(M +\e Du,K)\, dx \\
&\leq  C \e^2\int_{B_1} \dist^2(Du,T_M K)\, dx + o(\e^2).
\end{align*}
Hence, letting $X_\e=\e^{-1}(M-M_\e)$, we have
\begin{align*}
\int_{B_{1/2}}|X_\e -Du|^2\leq C \int_{B_1}\dist^2(Du,T_M K)\, dx + o(1).
\end{align*}
In particular $X_\e$ is bounded, and extracting a converging subsequence we obtain $X\in T_M K$ showing the validity of \eqref{eq:linweak} for $u\in C^1(\overline B_1;\R^2)$, and then by density for $u\in H^1(\overline B_1;\R^2)$.
\end{remark}

\begin{proof}[Proof of Lemma~\ref{l:exact}]
 Let $\ell=|K|$ and    $M\colon \R/\ell Z\to K$ be an arc-length parametrization of $K$, the ellipticity assumption ensures that $M'(t)$ is invertible for all $t\in\R$. Let $u\in H^1(B_1;\R^2)$ such that $Du\in K$ a.e., then $u$ is smooth by \cite{sverak93}, and since $B_1$ is simply connected there exists a smooth lifting $\theta\colon B_1\to\R$ such that $Du=M(\theta)$.
Using that  $\dv \cof(Du)=0$, where $\cof$ denotes the cofactor matrix, we find
$\cof (M'(\theta))\nabla\theta =0$, hence $\nabla\theta=0$ since $\cof(M'(\theta))$ is invertible. Therefore $Du$ is constant.
\end{proof}

\begin{proof}[Proof of Lemma~\ref{l:lin}] 
For $M\in K$ we denote by $P_M\colon\R^{2\times 2}\to\R^{2\times 2}$ the orthogonal projection onto $(T_M K)^\perp$, so that
\begin{align*}
\dist^2(X,T_M K) =|P_M X|^2.
\end{align*}
We denote by $P_M^{\alpha\beta jk}\in\R$ the coefficients of $P_M$, that is,
\begin{align*}
(P_MX)_{\alpha\beta}=\sum_{jk}P_M^{\alpha\beta jk}X_{jk}\qquad\forall X=(X_{jk})\in\R^{2\times 2}.
\end{align*}
Define $\mathbb P_M(i\xi)\in \mathcal L(\C^2; \C^{2\times 2})$ by 
\begin{align*}
(\mathbb P_M(i\xi)v)_{\alpha\beta} =\sum_{jk}P_M^{\alpha\beta jk}v_j\xi_k\qquad\forall \xi,v\in\C^2,
\end{align*}
so the differential operator $u\mapsto P_MDu$ has symbol $\mathbb P_M(i\xi)$,
 i.e.,
\begin{align*}
\lt(P_M\lt(Df\rt)\rt)_{\alpha\beta}=\frac{1}{2\pi}\int \lt(\mathbb{P}_M\lt(i\xi\rt) \hat{f}\lt(\xi\rt) \rt)_{\alpha \beta} e^{i x\cdot \xi} d\xi.
\end{align*}
We claim that $\mathbb P_M(\xi)$ has trivial kernel for all non-zero $\xi\in\C^2$. Let indeed $v\in\C^2$ such that $\mathbb P_M(\xi)v=0$. This implies that $P_M \re(v\otimes \xi)=P_M \im(v\otimes\xi)=0$, because the coefficients $P_M^{\alpha\beta jk}$ are real-valued. 
In other words, the real and imaginary parts of $v\otimes\xi$ both belong to $\ker P_M=T_M K$.
Since $T_M K$ is a one-dimensional subspace of $\R^{2\times 2}$ which doesn't contain any rank-one matrix, we have $T_M K =\R X_0$ for some invertible matrix $X_0$. 
Hence we deduce that $v\otimes \xi =\lambda X_0$ for some $\lambda\in\C$ and an invertible matrix $X_0\in T_M K $. But $v\otimes\xi$ has zero determinant, so $\lambda=0$ and we must have $v=0$. 
This proves that $\mathbb P_M(\xi)$ has trivial kernel.
Therefore we have the representation formula 
\cite[Theorem~4.1]{smith70} 
and the coercive inequality 
\cite[Theorem~8.15]{smith70} that follows from it,
\begin{align}\label{eq:linlow}
\int_{B_1}|Du|^2\, dx \leq C \int_{B_1}|P_MDu|^2\, dx + C \int_{B_1}|u|^2\, dx,
\end{align} 
for all $u\in H^1(B_1;\R^2)$.  
(In the notation of \cite{smith70}, $N=4$, $M=2$, and the index set $\lt\{1,2,3,4\rt\}$ for $j$ is in our case given by $\lt\{1,2\rt\}^2$, and we can take $m_j=1$ for $j\in\lbrace 1,2\rbrace^2$ and $l_i=0$ for $i\in\lt\{1,2\rt\}$.)
The constant $C>0$ in \eqref{eq:linlow} depends a priori on the fixed matrix $M\in K $. Denote by $C(M)$ the best possible constant in \eqref{eq:linlow}. Then for any $M,M'\in K $ we have
\begin{align*}
\int_{B_1}|Du|^2\, dx &\leq 2C(M) \int_{B_1}|P_{M'}Du|^2\, dx + C(M) \int_{B_1}|u|^2\, dx\\
&\quad + 2C(M)\| P_M-P_{M'}\|^{2} \int_{B_1}|Du|^2\, dx.
\end{align*}
For all $M\in K $, there exists $\delta(M)>0$ sufficiently small such that for all $M'\in  K \cap B_{\delta(M)}(M)$, we have $2C(M)\| P_M-P_{M'}\|^2<1/2$. It follows that
\begin{align*}
C(M')\leq \frac{2C(M)}{1-2C(M)\|P_M-P_{M'}\|^{2}}
<4C(M)
\qd
\forall M'\in  K \cap B_{\delta(M)}(M). 
\end{align*} 
By compactness this implies that $M\mapsto C(M)$ is bounded on $ K $, so we can take the constant $C$ in \eqref{eq:linlow} to depend only on $ K $.

Moreover, if $u\in H^1(B_1;\R^2)$ satisfies $P_MDu=0$ a.e., then
$Du=\lambda X_0$
for some $\lambda\in L^2(B_1;\R)$,
and the  distributional 
 identity $0=\dv \cof(Du)=\cof(X_0)\nabla\lambda$
 implies that $\lambda$ is constant, hence $Du\equiv X$ for some $X\in T_M K $. 

Therefore \eqref{eq:lin} follows from \eqref{eq:linlow} via a compactness argument: assume by contradiction the existence of sequences $M^k\in K $, and $u^k\in H^1(B_1;\R^2)$ such that
\begin{align*}
\inf_{X\in T_{M^k} K }\int_{B_1} |Du^k-X|^2 \,dx =1,\quad \int_{B_1}|P_{M^k}Du^k|^2\, dx\longrightarrow 0.
\end{align*}
Subtracting from $u^k$ its average and $X^kx$ for the matrix $X^k$ at which the infimum in the left-hand side is attained, we may in fact assume 
\begin{align*}
\int_{B_1} u^k \, dx =\int_{B_1} D u^k \cdot X\, dx =0\qquad\forall X\in T_{M^k} K ,
\end{align*}
and
\begin{align*}
 \int_{B_1} |Du^k|^2 \,dx =1,\quad \int_{B_1}|P_{M^k}Du^k|^2\, dx\longrightarrow 0.
\end{align*}
Thus we may extract subsequences (not relabeled) $u^k\to u$ weakly in $H^1(B_1;\R^2)$ and strongly in $L^2(B_1;\R^2)$, and $M^k\to M\in K $. It follows that $P_{M^k}Du^k\rightharpoonup P_{M}Du$ in $L^2(B_1;\R^{2\times 2})$, 
and thus by lower semicontinuity of the $L^2$ norm under weak convergence, we have
$P_{M}Du=0$ a.e., 
which implies $Du\equiv X$ for some $X\in T_M K $. 
Approximating $X$ by a sequence $X^k\in T_{M^k} K $ and using $Du^k\rightharpoonup Du$ in $L^2(B_1;\R^{2\times 2})$, we deduce that $0=\int_{B_1} Du \cdot X\, dx=|B_1||X|^2$, and thus $Du\equiv X=0$. 
Further $u$ satisfies $\int_{B_1} u\, dx=0$, which implies $u\equiv0$. Plugging $u^k$ into \eqref{eq:linlow} gives
\begin{align*}
	1=\int_{B_1}|Du^k|^2\, dx \leq C \int_{B_1}|P_{M^k}Du^k|^2\, dx + C \int_{B_1}|u^k|^2\, dx.
\end{align*} 
Passing to the limit as $k\to\infty$ and using the strong $L^2$ convergence, we have $1\leq C \int_{B_1} |u|^2\, dx=0$, which gives a contradiction.
\end{proof}

\section{Proof of Proposition~\ref{p:G}}\label{s:G}

We only prove the existence of $G_1$, the existence of $G_2$ is obtained by the same arguments. The proof relies on the properties of the conformal and anticonformal projections of $ K $ uncovered in \cite{zhang97,FS08}. For any $z_+,z_-\in\mathbb C$, we denote by
\begin{align*}
[z_+,z_-]=\left(\begin{array}{cc}\re z_+ & -\im z_+ \\ \im z_+ & \re z_+ \end{array}\right)
+
\left(\begin{array}{cc}\re z_- & \im z_- \\ \im z_- & -\re z_- \end{array}\right)\in\R^{2\times2},
\end{align*}
the $2\times 2$ matrix whose conformal, respectively anticonformal, part is represented by $z_+$, respectively $z_-$. For any $A\in\R^{2\times2}$, the decomposition $A=[z_+,z_-]$ is unique, and 
we have the identities
\begin{equation*}
	\det A = |z_+|^2 - |z_-|^2,
\quad
	|A|^2=2|z_+|^2+2|z_-|^2, \qd \|A\|=|z_+| + |z_-|,
\end{equation*}
where $|A|$ and $\|A\|$ denote the Hilbert-Schmidt and the operator norms of $A$, respectively. We denote by $p_+\colon [z_+,z_-]\mapsto z_+$ the projection onto the conformal part.

 Using these notations, the ellipticity assumption \eqref{eq:ellipt} is equivalent to
 \begin{align*}
2|z_+-z_+'|^2+2|z_- -z_-'|^2 
&
\leq C_* \left(|z_+-z_+'|^2+|z_- -z_-'|^2\right),
 \end{align*}
for all $[z_+,z_-],[z_+',z_-']\in K $. 
This implies that the curve $ K $ is $C_*$-elliptic in the sense of \cite[Def. 1]{FS08}, and it follows from the analysis in \cite{zhang97} (see also the proofs of Lemmas~1 and 2 in \cite{FS08}) that $p_+( K )$ is a Jordan curve and 
 \begin{align*}
  K  =\left\lbrace [z,H(z)] \colon z\in p_+( K )\right\rbrace,
 \end{align*}
 for some $k$-Lipschitz function $H\colon p_+( K )\to\C$ with $0\leq k=(C_*-1)/(C_*+1)<1$. 
Further the explicit formula $p_{+}(A)=(a_{11}+a_{22})/2-i(a_{12}-a_{21})/2$ for $A=(a_{ij})$ and the smoothness of $ K $ imply that $p_{+}( K )\subset\mathbb{C}$ is smooth.
 
 \begin{lemma}\label{l:extH}
 The function $H$ admits a smooth extension $H\colon\C\to\C$ which is $k$-Lipschitz for some (possibly larger) $0\leq k<1$.
 \end{lemma}
 \begin{proof}[Proof of Lemma~\ref{l:extH}]
We first fix, thanks to Kirszbraun's theorem, a $k$-Lipschitz extension $\widehat H\colon \C\to\C$. In the rest of the proof we modify $\widehat H$ to make it smooth while still agreeing with $H$ on $p_+( K )$, at the cost of slighlty increasing  its Lipschitz constant.
 
 Let $[z_+(t),z_-(t)]$, $t\in\R/\ell\Z$, denote a smooth arc-length parametrization of $ K $, so that $2|\dot z_+|^2+2|\dot z_-|^2=1$, where $\dot z_{\pm}=\frac{d}{dt}z_{\pm}$. 
 As $z_- =H(z_+ )$ and $H$ is $k$-Lipschitz, it follows that $|\dot z_-|\leq k|\dot z_+|$, so $|\dot z_+|^2\geq  \frac{1}{2(1+k^2)} >0$.
Therefore we may reparametrize and consider 
\begin{align*}
 K  =\left\lbrace [z(s),H(z(s))]\colon s\in \R/\ell_+\Z \right\rbrace,
\end{align*} 
with $z(s)$ an arc-length parametrization of $p_+( K )$ and $\ell_+$ its length, and the map
$s\mapsto H(z(s))$ is smooth by smoothness of $ K $. 

\medskip

For small enough $\delta>0$, the map
 \begin{align*}
 \varphi\colon \R/\ell_+\Z\times (-2\delta,2\delta) & \to \mathcal U_{2\delta }= \lbrace z\in\C\colon \dist(z,p_+( K ))<2\delta\rbrace\\
 (s,r) &\mapsto z(s) +r i\dot z(s),
 \end{align*}
 is a smooth diffeomorphism.
 We first modify $\widehat H$ by setting
 \begin{align*}
 \widetilde H =\widehat H\circ\Phi,\qquad
  \Phi(Z)=\begin{cases}
 z(s)+\lambda(r) i\dot z(s) 
 &\text{if }Z=\varphi(s,r)\in\mathcal U_{2\delta}, \\
 Z & \text{otherwise,}
 \end{cases}
 \end{align*}
 where $\lambda$ is the odd $(1-\delta)^{-1}$-Lipschitz function given for $r>0$ by
\begin{align*}
\lambda(r)=\begin{cases}
0 &\text{for } 0\leq r \leq \delta^2,\\
-\frac{\delta^2}{1-\delta} +\frac{r}{1-\delta} &\text{for }\delta^2< r\leq \delta,\\
r &\text{for } r>\delta.
\end{cases}
\end{align*} 
In particular we have
 \begin{align*}
 \widetilde H(Z)=H(z(s))\qquad \forall Z=\varphi(s,r)\in\mathcal U_{\delta^2},
 \end{align*}
 so   $\widetilde H$ is smooth in $\mathcal U_{\delta^2}$ (by smoothness of $s\mapsto H(z(s))$ and $\varphi^{-1}$) and agrees with $H$ on $p_+( K )$. 
 Note that by definition of $\varphi$ and $\lambda$ we have $\Phi(Z)=Z$ in $\C\setminus U_{\delta}$ and therefore $\Phi$ is Lipschitz in $\C$.
Since $D\varphi(s,0)\in SO(2)$ for all $s\in\R/\ell_+\Z$, we have $\| D\varphi\|\leq 1+C\delta$ on $\R/\ell_+\Z\times(-2\delta,2\delta)$. Further, we have $\| D(\varphi^{-1})\|=1$ on $p_+( K )$, and $\| D(\varphi^{-1})\|\leq 1+C\delta$ on $\mathcal U_{2\delta}$. Denoting by $\psi$ the $(1-\delta)^{-1}$-Lipschitz map $(s,r)\mapsto (s,\lambda(r))$, we write $\Phi(Z)=\varphi(\psi(\varphi^{-1}(Z)))$ and deduce that $\|D\Phi\|\leq 1+C\delta$ a.e. in $\mathcal U_{2\delta}$. This inequality is also true in the rest of $\mathbb C$ by definition of $\Phi$, so we conclude that $\widetilde H$ is $\tilde k$-Lipschitz in $\mathbb C$, with  $\tilde k = (1+C\delta)k<1$ for small enough $\delta>0$. Now $\delta$ is fixed and we define, for $\e\in (0,\delta^2/4)$,
 \begin{align*}
 H_\e(z)=\int_\C \widetilde H(z+\e  \chi(z) y)\rho(y)\, dy,
 \end{align*}
 for a smooth kernel $\rho\geq 0$ with support in $B_1$ 
 and $\int\rho(y) \, dy=1$, and some smooth cut-off function $\chi$ with
 $\mathbf 1_{\mathcal U_{\delta^2/4}}\leq 1- \chi\leq \mathbf 1_{\mathcal U_{\delta^2/2}}$. 
In $\mathcal U_{3\delta^2/4}$, the map $H_\e$ is smooth thanks to the smoothness of $\widetilde H$ in $\mathcal U_{\delta^2}$. 
In $\mathbb{C}\setminus \overline{\mathcal U}_{\delta^2/2}$, we have $H_\e(z)=\int_\C \widetilde H(z+\e y)\rho(y)\, dy$ is also smooth. 
Therefore the map $H_\e$ is smooth in $\mathbb{C}$. 
Further, $H_\e(z)=\widetilde H(z)$ for $z\in\mathcal U_{\delta^2/4}$, and thus agrees with $H$ on $p_+( K )$. 
Finally, denoting by $L$ the Lipschitz constant of $\chi$, we have
 \begin{align*}
 	|H_\e(z)-H_\e(z')|
 	&\leq \int_{B_1}|\widetilde H(z+\e\chi(z)y)-\widetilde H(z'+\e\chi(z')y)|\rho(y)\,dy\\
 	&
 	\leq \int_{B_1}\tilde k(1+\e L|y|)|z-z'|\rho(y)\,dy\leq\tilde k(1+\e L)|z-z'|.
 \end{align*}
 So $H_\e$ is $k_\e$-Lipschitz with $k_\e \leq \hat k (1+\e L)<1$ for small enough $\e$.
  \end{proof}
 
 \begin{lemma}\label{l:bij}
 The map $F\colon z\mapsto \bar z +H(z)$ is a 
 smooth diffeomorphism from $\C$ onto $\C$. Moreover $F$ is $(1+k)$-Lispchitz and $F^{-1}$ is $(1-k)^{-1}$-Lispchitz.
 \end{lemma}
  \begin{proof}[Proof of Lemma~\ref{l:bij}]
  For any $w\in\C$ the equation 
  \begin{align*}
  w=\bar z +H(z) \quad\Leftrightarrow\quad z =\bar w - \bar H(z),
  \end{align*}
  admits a unique solution $z\in\C$ thanks to the fixed point theorem, since $z\mapsto \bar w-\bar H(z)$ is $k$-Lipschitz and $0\leq k<1$. This shows that $F$ is bijective. The inequalities
  \begin{align*}
  (1-k)|z-z'| \leq |F(z)-F(z')| \leq (1+k) |z-z'|,
  \end{align*}
  follow directly from the fact that $H$ is $k$-Lipschitz and imply the announced Lipschitz constants of $F$ and $F^{-1}$. 
The inverse $F^{-1}$ is smooth thanks to the Inverse Function Theorem, since 
$DF(z)\colon h\mapsto \bar h +DH(z)h$ is invertible for all $z\in\C$ because $\|DH\| <1$.
 \end{proof}
 
 \begin{proof}[Proof of Proposition~\ref{p:G} completed]
 Then, identifying $\C$ with $\R^2$, we define
 \begin{align*}
 G_1(A)=\overline{F^{-1}(A)} -H(F^{-1}(A)),
 \end{align*}
 so that a short calculation shows that
 \begin{align*}
 [z,H(z)]=\left(\begin{array}{c} A \\ iG_1(A)\end{array}\right)\qquad\text{for }A=F(z).
 \end{align*}
 The map $G_1$ is smooth and globally Lipschitz with Lipschitz constant $\Lambda=(1+k)/(1-k)$.
 Moreover, 
for all $A=F(z)$, $A'=F(z')$, we have
 \begin{align*}
 (G_1(A)-G_1(A'))\cdot (A-A') &
 =\det
 \left(\begin{array}{c} A-A' \\ i(G_{1}(A)-G_{1}(A'))\end{array}\right)\\
 & =\det ([z-z',H(z)-H(z')] )\\
 &=|z-z'|^2-|H(z)-H(z')|^2 \\
 &\geq (1-k^2)|z-z'|^2. 
 \end{align*}
 Since  $F$ is $(1+k)$-Lispchitz, this implies
 \begin{align*}
  (G_{1}(A)-G_{1}(A'))\cdot (A-A') 
  \geq \lambda |A-A'|^2,
 \end{align*}
for $\lambda=(1-k^2)/(1+k)^{2}= (1-k)/(1+k)>0$.
This concludes the proof of Proposition~\ref{p:G}.
\end{proof}

\section{Proof of Theorem~\ref{t:main}}\label{s:proof}

\noindent\textbf{Step 1.} We may assume that $u$ is Lipschitz, thanks to the truncation argument of \cite[Proposition~A.1]{FJM02}. 
Let $ K \subset B_{R}\subset\R^{2\times 2}$ for some $R>0$. Then for all $X\in\R^{2\times 2}$ with $|X|>2R$, we have $|X|\leq 2\dist(X, K )$. 
An application of \cite[Proposition~A.1]{FJM02} with $\lambda=2R$ gives $v:B_1\to\R^2$ satisfying
\begin{align*}
	\|Dv\|_{L^{\infty}(B_1)}&\leq 2C_0R,\\
	\int_{B_1}|Du-Dv|^2\,dx &\leq C_0\int_{\{x\in B_1: |Du|>2R\}}|Du|^2\,dx\\
	&\leq 4C_0\int_{B_1}\dist^2(Du, K )\,dx,
\end{align*}
for some constant $C_0$ (depending only on $B_1$). If there exists $M\in K $ such that 
\begin{equation*}
	\int_{B_{1/2}}|Dv-M|^2\,dx\leq \wt C( K )\int_{B_1}\dist^2(Dv, K )\,dx,
\end{equation*}
then repeated applications of the triangle inequality give
\begin{align*}
	\int_{B_{1/2}}|Du-M|^2\,dx&\leq\lt(16\wt C( K )C_0+ 4\wt C( K )+8C_0\rt)\int_{B_1}\dist^2(Du, K )\,dx.
\end{align*}
Thus if Theorem~\ref{t:main} holds for all Lipschitz mappings $v$ for some constant $\wt C( K )$, then it is also valid for all $H^1$ mappings with $C( K )=16\wt C( K )C_0+ 4\wt C( K )+8C_0$.

\medskip

\noindent\textbf{Step 2.} We may assume in addition that $u\in C^2(B_1;\R^2)$ solves
\begin{align}\label{eq:eqGu}
\dv G_1(D u_1) =\dv G_2(D u_2)=0\qquad\text{ in }B_1.
\end{align}
Consider indeed $w\in C^2(B_1)$ such that $w=u$ on $\partial B_1$ and
\begin{align*}
\dv G_1(D w_1)=\dv G_2(D w_2)=0\qquad\text{ in }B_1.
\end{align*}
The existence of such $w$ is guaranteed by the ellipticity of the equation $0=\dv G_j(D w_j)=\mathrm{tr}( DG_j(Dw_j)D^2w_j)=0$, invoking e.g. \cite[Theorem~12.5]{GT}: 
the inequality $\lambda|\xi|^2 \leq DG_j(A)\xi\cdot\xi\leq \Lambda |\xi|^2$, valid for all $A,\xi\in\R^2$ thanks to Proposition~\ref{p:G}, ensures that the eigenvalues of the symmetric part   $[DG_j(A)]_s=(DG_j(A)+DG_j(A)^T)/2$ of $DG_j(A)$ are bounded above and below (since $DG_j(A)\xi\cdot\xi=[DG_j(A)]_s\xi\cdot\xi$ for all $\xi\in\R^2$) and in particular condition (ii) in \cite[Theorem~12.5]{GT} is satisfied.
 Letting $v=u-w$ and using the uniform monotonicity of $G_1$ we find
\begin{align*}
\lambda\int_{B_1}|D v_1|^2\, dx & \leq \int_{B_1} (G_1(D u_1)-G_1(D w_1))\cdot Dv_1\, dx.
\end{align*}
Since $\dv G_1(Dw_1)=0$ and $\dv (iDu_2)=0$ we rewrite this as
\begin{align*}
\lambda\int_{B_1}|D v_1|^2\, dx & \leq \int_{B_1} (G_1(Du_1)+iDu_2)\cdot D v_1\, dx\\
&\leq \frac{1}{2\lambda}\int_{B_1} |G_1(Du_1)+iDu_2|^2\,dx + \frac{\lambda}{2}\int_{B_1}|Dv_1|^2\, dx,
\end{align*}
and infer
\begin{align*}
\int_{B_1}|Dv_1|^2 \, dx &\leq \frac{1}{\lambda^2} \int_{B_1}|G_1(Du_1)+iDu_2|^2\,dx.
\end{align*}
According to Proposition~\ref{p:G} the function $M\mapsto G_1(A)+iB$, where $A,B$ denote the first and second row of the matrix $M$, vanishes on $ K $. Since that function is Lipschitz we deduce that $|G_1(Du_1)+iDu_2|\leq C \dist(Du, K )$, and therefore
\begin{align*}
\int_{B_1}|Dv_1|^2 \, dx &\leq C \int_{B_1}\dist^2(Du, K )\,dx.
\end{align*}
Applying a similar argument to $v_2$ we obtain
\begin{align*}
\int_{B_1}|Dv|^2 \, dx \leq C \int_{B_1}\dist^2(Du, K )\,dx.
\end{align*}
Recalling that $v=u-w$ and using the triangle inequality we deduce
\begin{align*}
\int_{B_1}\dist^2(Dw, K )\, dx & \leq C \int_{B_1}\dist^2(Du, K )\, dx,\\
\int_{B_{1/2}}|Du-M|^2\,dx & 
\leq 2 \int_{B_{1/2}}|Dw-M|^2\, dx + C\int_{B_1}\dist^2(Du, K )\, dx.
\end{align*}
As a consequence, if Theorem~\ref{t:main} is valid for $w$ then we obtain it for $u$. This proves Step~2.

\medskip

\noindent\textbf{Step~3.} As $u_j\in C^2(B_1)$ satisfies \eqref{eq:eqGu}, it is a weak solution of
\begin{equation*}
	\dv\lt(\partial_i(G_j(Du_j))\rt) =\dv\lt(DG_j(Du_j)D(\partial_i u_j)\rt)= 0\qquad\text{for }i=1, 2. 
\end{equation*}
Invoking the De Giorgi-Nash estimates \cite{degiorgi} 
(see e.g. \cite[Theorem~4.11]{hanlin} for the precise statement we use here) 
for $\partial_i u_j$, $i, j\in\{1, 2\}$, we obtain
\begin{equation*}
	\|Du\|_{C^{\alpha}(\overline B_{1/2})} \leq C\|Du\|_{L^2(B_1)}, 
\end{equation*}
for some $\alpha>0$ and some constant $C=C( K )$. 
Thanks to this estimate and the exact rigidity obtained in Lemma~\ref{l:exact}, we may argue exactly as in \cite[Lemma~4.5]{FZ05} to deduce that 
\begin{align}\label{eq:Estrho}
\inf_{M\in K } \| Du-M\|_{L^\infty(B_{1/2})} \leq \rho\left(\int_{B_1}\dist^2(Du, K )\, dx \right),
\end{align}
for some function $\rho$  depending only on $ K $ and satisfying $\rho(\e)\to 0$ as $\e\to 0$.

\medskip

\noindent\textbf{Step~4.} We finally combine Step~3 with the linearized estimate of Lemma~\ref{l:lin}, to obtain our main estimate \eqref{eq:mainest}. 
The basic idea, as in \cite{FJM02,FZ05}, is to linearize $\dist^2(\cdot, K )$ around $M_0\in K $ such that $|Du-M_0|$ is uniformly small. 
When doing so, \eqref{eq:mainest} formally turns into the linearized estimate \eqref{eq:lin} of Lemma~\ref{l:lin}, and it remains to control the error terms.
Due to the modification of $u$ arising from the translation $X\in T_M K $ in the left-hand side of the linearized estimate \eqref{eq:lin},
it is not directly obvious that the error terms are negligible.
In \cite{FJM02} this problem is absent 
because their equivalent of \eqref{eq:Estrho} comes with an explicit $\rho(\e)=C\e^{\frac 14}$.
In \cite{FZ05} it is taken care of via a 
topological degree argument \cite[Proposition~4.7]{FZ05} 
(see also \cite{reshetnyak94}) which allows to avoid the translation.
Here  we present an alternative argument relying on elementary estimates.

We assume without loss of generality that 
\begin{align}\label{eq:eps}
\int_{B_1}\dist^2(Du, K )\, dx =\e \leq\e_0,
\end{align}
where $\e_0=\e_0( K )$ is to be chosen in the course of the proof. 
If \eqref{eq:eps} is not valid then \eqref{eq:mainest} is automatically satisfied for a large enough constant $C$ because the left-hand side of\eqref{eq:mainest} is bounded thanks to Step~1.

We fix $\delta_0>0$ depending only on $ K $, such that the nearest-point projection $\Pi_ K $ onto $ K $ is uniquely defined and smooth in  the neighborhood $\mathcal N_{2\delta_0}( K )$.
We first choose $\e_0$ small enough that  
$\rho(\e_0)\leq \delta_0$, so
 thanks to \eqref{eq:Estrho} the projection $\Pi_ K (Du)$ is well-defined.

We claim that, for every $M\in K $, there exists $Y_M\in K $ such that
\begin{align}\label{eq:linearize}
\int_{B_{1/2}}|Du-Y_M|^2 \, dx
&
\leq  C\int_{B_{1/2}} \dist^2(Du,  K  )\, dx
 + C \int_{B_{1/2}} |Du-M|^4\, dx,\\
 \text{and }\quad
 |M-Y_M|^2 &
 \leq C \int_{B_{1/2}}|Du-M|^2\, dx.
 \nonumber
\end{align}
Here and in the rest of this proof we denote by $C>0$ a generic constant depending only on $ K $.

To prove \eqref{eq:linearize}, we first invoke
 Lemma~\ref{l:lin}, according to which we have
\begin{align*}
\inf_{X\in T_M K }\int_{B_{1/2}}|Du-M-X|^2\, dx & \leq  C\int_{B_{1/2}} \dist^2(Du-M,T_M  K  )\, dx.
\end{align*}
Choosing $X=X_M\in T_M K $ attaining the infimum in the left-hand side, we obtain
\begin{align}\label{eq:linearize1}
\int_{B_{1/2}}|Du-M-X_M|^2\, dx & \leq  C\int_{B_{1/2}} \dist^2(Du-M,T_M  K  )\, dx.
\end{align}
Moreover the minimizing property of $X_M$ implies 
 that $\int_{B_{1/2}}(Du-M-X_M)\, dx$ is orthogonal to $X_M$. We deduce
\begin{align*}
|X_M|^2
& = \fint_{B_{1/2}} X_M \cdot (Du-M)\, dx \leq \frac 12 |X_M|^2 +C \int_{B_{1/2}} |Du-M|^2\, dx,
\end{align*}
and therefore
\begin{align}\label{eq:XM}
|X_M|^2\leq C \int_{B_{1/2}} |Du-M|^2\, dx.
\end{align}
Recalling from the proof of Lemma~\ref{l:lin} that $P_M$ denotes the orthogonal projection onto $(T_M K )^\perp$, we estimate the integrand in the right-hand side of \eqref{eq:linearize1} using
\begin{align*}
\dist(Du-M,T_M  K  ) & =|P_M(Du-M)|\\
&
\leq |Du-\Pi_ K (Du)| + |\Pi_ K (Du)-M-(I-P_M)(Du-M)| \\
&\leq |Du-\Pi_ K (Du)| + C |Du-M|^2.
\end{align*}
The last inequality follows from the fact that $I-P_M=D\Pi_ K (M)$. Since 
$|Du-\Pi_ K (Du)|=\dist(Du, K )$, plugging this into \eqref{eq:linearize1} we find
\begin{align}\label{eq:linearize2}
\int_{B_{1/2}}|Du-M-X_M|^2\, dx 
 & \leq  C\int_{B_{1/2}} \dist^2(Du,  K  ) \, dx
 + C \int_{B_{1/2}} |Du-M|^4\, dx
\end{align}
Now we may choose $Y_M\in K $ such that
\begin{align}\label{eq:YM}
|M+X_M -Y_M|\leq C |X_M|^2.
\end{align}
Indeed, if $|X_M|\leq \delta_0$ then one can simply take $Y_M=\Pi_{ K }(M+X_M)$ and use the fact that $D\Pi_ K (M)X_M=(I-P_M)X_M=X_M$, and if $|X_M|\geq \delta_0$ one may take $Y_M=M$. From \eqref{eq:linearize2} and \eqref{eq:YM} we infer
\begin{align*}
\int_{B_{1/2}}|Du-Y_M|^2 \, dx
&
\leq  C\int_{B_{1/2}} \dist^2(Du,  K  )\, dx
 + C \int_{B_{1/2}} |Du-M|^4\, dx +C |X_M|^4.
\end{align*}
Using \eqref{eq:XM} and Cauchy-Schwarz to estimate the last term, we deduce the first inequality in \eqref{eq:linearize},
and the estimate on $|M-Y_M|$ in  \eqref{eq:linearize} follows from \eqref{eq:YM} and \eqref{eq:XM} (taking into account that $|X_M|\leq C$ thanks to \eqref{eq:XM} and Step~1).

Our goal is to find $M\in K $ for which we can discard the last term in \eqref{eq:linearize}.
To that end, we apply \eqref{eq:linearize} to define by induction
a sequence $(M_j)\subset K $ satisfying
\begin{align*}
M_{j+1}=Y_{M_j}\qquad\forall j\geq 0,
\end{align*}
and $M_0$ to be chosen later. According to \eqref{eq:linearize} we have
\begin{align*}
\int_{B_{1/2}}|Du-M_{j+1}|^2 \, dx
&
\leq  C\int_{B_{1/2}} \dist^2(Du,  K  )\, dx
 + C \int_{B_{1/2}} |Du-M_j|^4\, dx,\\
 \text{and }\quad
 |M_j-M_{j+1}|^2 &
 \leq C \int_{B_{1/2}}|Du-M_j|^2\, dx,
\end{align*}
Therefore, setting
\begin{align}\label{eq:qsj}
q_j=\int_{B_{1/2}}|Du-M_j|^2\quad\text{and}\quad s_j=\sup_{B_{1/2}}|Du-M_j|,
\end{align}
and recalling the definition of $\e=\int_{B_{1/2}}\dist^2(Du, K )\, dx$ \eqref{eq:eps}, we find
\begin{align}\label{eq:qsj+1}
q_{j+1} \leq C \e + C  s_j^2 q_j
\quad\text{and}\quad
s_{j+1} \leq  s_j + C q_j^{\frac 12}.
\end{align}
We wish to find $j_0\geq 0$ such that
\begin{align}\label{eq:qj0}
q_{j_0} \leq K\e,
\end{align}
for some constant $K>0$ depending only on $ K $. Recalling the definitions of $q_j$ \eqref{eq:qsj} and $\e$ \eqref{eq:eps}, this concludes the proof of \eqref{eq:mainest}.
We prove by induction on $j_0$ that, if $s_0\leq 1/(2K)$ 
 and 
\begin{align}\label{eq:indhyp1}
q_{j}> K\e\qquad\text{for }j=0,\ldots,j_0,
\end{align}
then
\begin{align}\label{eq:indhyp2}
q_j\leq (Ks_0)^{2j}s_0^2 \qquad\text{for }j=0,\ldots,j_0.
\end{align}
The constant $K>0$ depends only on $ K $ and will be adjusted in the course of the proof.
Initialization at $j_0=0$ is obvious since $ |B_{1/2}|=\pi/4\leq 1$.
Next we assume \eqref{eq:indhyp1}-\eqref{eq:indhyp2} and prove that either \eqref{eq:qj0} or \eqref{eq:indhyp2} is satisfied at $j_0+1$.
Combining \eqref{eq:indhyp2} with the second inequality in \eqref{eq:qsj+1}, we find
\begin{align*}
s_{j_0}\leq s_0 + C  s_0 \sum_{j=0}^{j_0-1}(Ks_0)^{j} \leq (1+C  )s_0,
\end{align*}
provided $2Ks_0\leq 1$. 
Therefore the first inequality in \eqref{eq:qsj+1} gives
\begin{align*}
q_{j_0+1} \leq C\e +C(1+C)^2 s_0^2 q_{j_0},
\end{align*}
and using \eqref{eq:indhyp2} we obtain
\begin{align*}
q_{j_0+1} &\leq C\e +C(1+C)^2  s_0^4 (Ks_0)^{2j_0}  \\
& =C\e +\frac{C(1+C)^2 }{K^2} (Ks_0)^{2j_0+2}s_0^2.
\end{align*}
Therefore we have
\begin{align*}
\text{either }q_{j_0+1}\leq 2C\e,
\quad\text{or}\quad
q_{j_0+1}\leq 2\frac{C(1+C)^2}{K^2} (Ks_0)^{2j_0+2}s_0^2.
\end{align*}
In the first case \eqref{eq:qj0} is satisfied for $j=j_0+1$ provided $K\geq 2C$, and in the second case \eqref{eq:indhyp2} is satisfied at $j_0+1$, provided 
$K\geq (2C)^{1/2}(1+C)$. 
This concludes the proof of \eqref{eq:indhyp2} by induction.

As a consequence, assuming by contradiction that there is no $j_0\geq 0$ satisfying \eqref{eq:qj0}, we would have
\begin{align*}
q_j > K\e\quad\text{and}\quad q_j\leq (Ks_0)^{2j}s_0^2\qquad\forall j\geq 0.
\end{align*} 
 As $s_0\leq 1/(2K)$, this implies $q_j\to 0$ and, passing to the limit in $q_j>K\e$ we deduce $\e=0$. 
But in that case by Lemma~\ref{l:exact} we have $Du\equiv M_0$ and this choice of $M_0$ gives $q_0=0$.

We conclude that there exists $j_0=j_0(\e)$ satisfying \eqref{eq:qj0}, if we can choose $M_0$ in such a way that $s_0\leq 1/(2K)$. But thanks to \eqref{eq:Estrho} we do have $s_0\leq\rho(\e)\leq 1/(2K)$ provided $\e_0$ is   small enough.
\qed

\section{A $3\times 3$ counter-example}\label{s:3x3}

In this section
we prove that the two-dimensional setting of Theorem~\ref{t:main} is optimal in the following sense: 
 a connected $1$-submanifold of $\R^{3\times 3}$ which has no rank-one connection  and is elliptic 
may not satisfy Sverak's compactness result \cite{sverak93}, and even less a quantitative rigidity  estimate. 

\begin{proposition}\label{p:3x3T4}
There exists a closed compact $1$-submanifold $\Pi \subset\R^{ 3\times 3}$ which is elliptic and has no rank-one connection, but contains a $\mathcal T_4$ configuration.
\end{proposition}

By a known construction, see e.g.  \cite[Theorem~3.1]{BFJK94}, Proposition~\ref{p:3x3T4} implies the existence of a sequence of   maps $u_k\colon \R^3\to\R^3$ such that 
\begin{align*}
\int_{B_1}\dist^2(Du_k,\Pi)\, dx\to 0\qquad\text{as }k\to 0,
\end{align*}
 but $(Du_k)$ is not precompact in $L^2(B_{1/2})$. 
 In particular, one certainly cannot hope for a quantitative estimate
 \begin{align*}
 \inf_{M\in\Pi}\int_{B_{1/2}}|Du_k-M|^2\, dx \leq \rho\left( \int_{B_1} \dist^2(Du_k,\Pi)\, dx\right),
 \end{align*}
for any function $\rho(\e)\to 0$ as $\e\to 0$.

\medskip

Proposition~\ref{p:3x3T4} is a consequence of the construction below.
Let $a>0$ and define matrices $T_1,T_2,T_3,T_4$ by
\begin{align*}
T_1 &
=-T_3
=\lt(\begin{array}{ccc}  1+a & 0 & 0 \\ 0 & 1 & 0\\ 1+a & 0  & 0\end{array}\rt), 
&
T_2
&=
-T_4
=\lt(\begin{array}{ccc}  -1 & 0 & 0 \\ 0 & 1+a & 0\\ -1 & 0 & 0 \end{array}\rt),
\end{align*}
%
and $C_1,C_2,C_3,C_4$ by
\begin{align*}
C_1
&=
-C_3
=\lt(\begin{array}{ccc}  1 & 0 & 0 \\ 0 & 1 & 0\\ 1 & 0 & 0  \end{array}\rt),
& 
C_2
&
=-C_4
=\lt(\begin{array}{ccc}  -1 & 0 & 0  \\ 0 & 1 & 0 \\ -1 & 0 & 0  \end{array}\rt).
\end{align*}
%
We have that $T_k-C_k$ is rank-one for $k=1,2,3,4$, and 
(with the convention that $C_5=C_1$)
%
\begin{align*}
&T_k-\frac{2+a}{a}\lt(T_k-C_k\rt)=C_{k+1} \qquad\text{ for }k=1,2,3,4,
\end{align*}
so $\lbrace T_1,T_2,T_3,T_4\rbrace$ forms a $\mathcal T_4$ configuration.
Next we construct a curve $\Pi=\Pi_a$ as in Proposition~\ref{p:3x3T4}, which contains this $\mathcal T_4$ configuration.
Let $\theta_a=\arctan (1/(1+a))$, so
\begin{align*}
\cos\theta_a &=\frac{1+a}{r_a},\qquad\sin\theta_a =\frac{1}{r_a},\qquad
r_a=\sqrt{1+(1+a)^2}.
\end{align*}
Let $\rho\colon \mathbb{R}\rightarrow \mathbb{R}$ be a smooth monotonically increasing function to be determined later that satisfies 
\begin{itemize}
\item $\rho\lt(\theta+2\pi\rt)=\rho\lt(\theta\rt)+2\pi$ for $\theta\in\R$,
\item $\rho\lt(\hat \theta_k\rt)=\hat \theta_k$ for
$\hat \theta_k =\theta_a +(k-1)\frac\pi 2$, $k=1,2,3,4$.
\end{itemize}
Define $\Pi_a=\Gamma_a(\R/2\pi\Z)$, with
\begin{align}
\label{eqb22.4}
\Gamma_a(\theta):=\lt(\begin{array}{ccc}  r_a \cos\lt(\theta\rt) & -\sin\lt(8\theta-8\theta_a\rt)  &  \sin\lt(6\rho\lt(\theta\rt)-6\theta_a\rt)   \\ \sin\lt(6\theta-6\theta_a\rt) &  r_a \sin\lt(\theta\rt) & \sin\lt(8\rho\lt(\theta\rt)-8\theta_a\rt) \\  r_a \cos\lt(\theta\rt) &    \sin\lt(8\theta-8\theta_a\rt)  &  \sin\lt(6\rho\lt(\theta\rt)-6\theta_a\rt) \end{array}\rt)
\end{align}
Then we have
\begin{align*}
T_k=\Gamma_a(\hat \theta_k)\text{ for }k=1,2,3,4,
\end{align*}
so $\Pi_a$ contains the  $\mathcal T_4$ configuration $\lbrace T_1,T_2,T_3,T_4\rbrace$. 
Next we adjust the parameter $a>0$ and the function $\rho$ in order to ensure that $\Pi_a$ has no rank-one connection and is elliptic.

\medskip

\paragraph{Notation.} With the $M_{i_1i_2,j_1j_2}$ minor we mean the determinant of the $2\times 2$ submatrix corresponding to the rows $i_1,i_2$ and columns $j_1,j_2$.

\medskip

\begin{lemma}
\label{l:Piaellipt}
If $a>0$ is such that $\theta_a\notin\frac{\pi}{48}\Z$, the curve $\Pi_a$ is elliptic, i.e. $\mathrm{Rank}\:\Gamma_a'(\theta) > 1$ for all $\theta\in\R$. 
\end{lemma}
\begin{proof} 
The derivative $\Gamma_a'$ is given by
\begin{align*}
\Gamma_a'(\theta):=\lt(\begin{array}{ccc}  -r_a \sin\lt(\theta\rt) & -8\cos\lt(8\theta-8\theta_a\rt)  &  6\rho'(\theta)\cos\lt(6\rho\lt(\theta\rt)-6\theta_a\rt)   \\ 6\cos\lt(6\theta-6\theta_a\rt) &  r_a \cos\lt(\theta\rt) &    8\rho'(\theta)\cos\lt(8\rho\lt(\theta\rt)-8\theta_a\rt)   \\  -r_a \sin\lt(\theta\rt) &    8\cos\lt(8\theta-8\theta_a\rt)  &  6\rho'(\theta)\cos\lt(6\rho\lt(\theta\rt)-6\theta_a\rt) \end{array}\rt)
\end{align*}
Assume  $\mathrm{Rank}\:\Gamma_a'(\theta)\leq 1$ for some $\theta\in\R$.
Then calculating the $M_{12,12}$ minor we have 
\begin{align*}
 -r_a^2 \sin\lt(\theta\rt)\cos\lt(\theta\rt) +48 \cos\lt(8\theta-8\theta_a\rt)\cos\lt(6\theta-6\theta_a\rt)=0
\end{align*}
and calculating the $M_{23,12}$ minor we have 
\begin{align*}
48 \cos\lt(8\theta-8\theta_a\rt)\cos\lt(6\theta-6\theta_a\rt)+ r_a^2
\sin\lt(\theta\rt)\cos\lt(\theta\rt)  =0.
\end{align*}
Adding  and substracting these two equations we obtain that 
\begin{align*}
  \sin\lt(\theta\rt)\cos\lt(\theta\rt)=0\text{ and }\cos\lt(8\theta-8\theta_a\rt)\cos\lt(6\theta-6\theta_a\rt)=0.
\end{align*}
The first equality implies  $\theta\in\frac\pi 2\Z$, and then the second equality becomes
\begin{align*}
\cos(8\theta_a)\cos(6\theta_a)=0,
\end{align*}
which is impossible for $\theta_a\notin \frac{\pi}{48}\Z$.
\end{proof}

\medskip

\begin{lemma}
\label{l:rho}
For any $a>0$ such that $\theta_a\notin\frac\pi{24}\Z$, for any  $\e>0$,   there exists a smooth monotonic function $\rho\colon \mathbb{R}\rightarrow \mathbb{R}$  with the following properties 
\begin{itemize}
\item $\rho\lt(\theta+2\pi\rt)=\rho\lt(\theta\rt)+2\pi$ for $\theta\in\R$,
\item $\rho\lt(\hat \theta_k\rt)=\hat \theta_k$ for
$\theta_k =\hat \theta_a +(k-1)\frac\pi 2$, $k=1,2,3,4$.
\item For any $\theta, \theta'\in [0,2\pi)\cap \frac{\pi}{24}\mathbb{Z}$, we have $\rho(\theta)-\rho(\theta')\notin \frac\pi {12}\Z$
\item $\sup_{x\in \mathbb{R}}\lt|\rho(x)-x\rt|<\e$
\end{itemize}
\end{lemma}
\begin{proof} 
Let $\delta=\frac 12 \dist(\theta_a,\frac\pi{24}\Z) = \frac 12 \dist(\lbrace \hat\theta_k\rbrace,\frac\pi{24}\Z) >0$ and fix a smooth function $\varphi$ such that $\supp\varphi\subset (-\delta,\delta)$, $0\leq\varphi \leq\varphi(0)=1$ and $|\varphi'|\leq 2/\delta$.
Define $\rho$ on $[0 ,2\pi)$ by setting
\begin{align*}
\rho(\theta)=\theta +\sum_{j=1}^{47}t_j\,\varphi\left(\theta-j\frac{\pi}{24}\right)\qquad\text{ for }\theta\in [0 ,2\pi),
\end{align*}
where $t_1,\ldots,t_{47}\in (-\eta,\eta)$ are to be fixed later and $\eta=\min(\e,\delta/2)$.
The choice of $\delta>0$ ensures that $\rho(\hat\theta_k)=\hat\theta_k$,  also since $|t_j|<\e$ we have $|\rho-id|<\e$, and 
finally since $|t_j|<\delta/2$ we have $\rho'>0$ on $[0,2\pi )$. Moreover the function $\rho-id$ is identically zero near $0$ and $2\pi$, so it can be extended to a smooth $2\pi$ periodic function, thus yielding  a smooth monotonic extension $\rho\colon \R\to\R$ satisfying $\rho(\theta+2\pi)=\rho(\theta)+2\pi$ for $\theta\in\R$.
It remains to argue that we can pick $t_1,\ldots,t_{47}\in [0,\eta)$ to ensure that the third condition in Lemma~\ref{l:rho} is satisfied.

Denote $t_0=0$. 
By induction, we may for each $j=1,\ldots,47$ choose $t_j\in (-\eta,\eta)$ to ensure that  
\begin{align*}
\rho(j\pi/24)-\rho(\ell\pi/24)=t_j-t_\ell + (j-\ell)\pi/24 \notin \frac\pi {12}\Z
\end{align*}
for all $\ell\in\lbrace 0,\ldots,j-1\rbrace$. This is possible because at each step there is only a discrete set of values of $t_j$ to avoid.
\end{proof}

\begin{lemma} 
\label{L:PianoR1}
If $a>0$ is such that $\theta_a\notin \frac{\pi}{48}\Z$,
 $\e>0$ is small enough
 and $\rho$ is as in Lemma~\ref{l:rho},
  then 
the curve $\Pi_a\subset\R^{3\times 3}$ does not contain Rank-$1$ connections. 
\end{lemma}
\begin{proof}
Note as in Lemma~\ref{l:Piaellipt} that the assumption $\theta_a\notin\frac{\pi}{48}\Z$ implies
\begin{align*}
\cos\lt(8\theta_a\rt)\not=0\quad\text{and}\quad  \cos\lt(6\theta_a\rt)\not=0.
\end{align*}
We assume that there exist
 $\theta\neq\theta'\in\R/2\pi\Z$ such that
\begin{align}
\label{eqb1}
\mathrm{Rank}\lt(\Gamma_a(\theta)-\Gamma_a(\theta')\rt)=1,
\end{align}
and we  obtain  a contradiction. We do this in several steps.

\textit{Step 1.} 
We have
\begin{align}
\label{eqc3}
\theta+\theta'\in\pi\Z
\quad\text{ and }
\quad
\theta-\theta'\in\frac\pi 3\Z \cup\frac\pi 4\Z.
\end{align}

\em Proof of Step 1. \rm
From \eqref{eqb22.4} calculating the $M_{12,12}$ minor we have 
\begin{align}
\label{eqb2}
0&= r_a^2\lt(\cos\lt(\theta\rt)-\cos\lt(\theta'\rt)\rt)\lt( \sin\lt(\theta \rt)-\sin\lt(\theta'\rt) \rt)\nn\\
&\quad+\lt( \sin\lt(8\theta -8\theta_a\rt)- \sin\lt(8\theta'-8\theta_a\rt) \rt) \lt( \sin\lt(6\theta -6\theta_a\rt)- \sin\lt(6\theta'-6\theta_a\rt) \rt)\nn\\
&=-4  r_a^2 \sin\lt(\frac{\theta +\theta'}{2}\rt)\cos\lt(\frac{\theta +\theta'}{2}\rt) \sin^2\lt(\frac{\theta -\theta'}{2}\rt)\nn\\
&\quad +4   \sin\lt(3\lt(\theta -\theta'\rt)\rt)\sin\lt(4\lt(\theta -\theta'\rt)\rt)\cos\lt(3\lt(\theta +\theta'\rt)-6 \theta_a\rt)\cos\lt(4\lt(\theta +\theta'\rt)-8 \theta_a\rt).
\end{align}
And calculating the  $M_{23,12}$ minor we have 
\begin{align}
\label{eqb50}
0&=- r_a^2\lt(\cos\lt(\theta\rt)-\cos\lt(\theta'\rt)\rt)\lt( \sin\lt(\theta \rt)-\sin\lt(\theta'\rt) \rt)\nn\\
&\quad+\lt( \sin\lt(8\theta -8\theta_a\rt)- \sin\lt(8\theta'-8\theta_a\rt) \rt) \lt( \sin\lt(6\theta -6\theta_a\rt)- \sin\lt(6\theta'-6\theta_a\rt) \rt)\nn\\
&=4  r_a^2  \sin\lt(\frac{\theta +\theta'}{2}\rt)\cos\lt(\frac{\theta +\theta'}{2}\rt) \sin^2\lt(\frac{\theta -\theta'}{2}\rt)\nn\\
&\quad +4   \sin\lt(3\lt(\theta -\theta'\rt)\rt)\sin\lt(4\lt(\theta -\theta'\rt)\rt)\cos\lt(3\lt(\theta +\theta'\rt)-6 \theta_a\rt)\cos\lt(4\lt(\theta +\theta'\rt)-8 \theta_a\rt).
\end{align}
Adding and substracting \eqref{eqb2} and \eqref{eqb50}   we obtain the   equations 
\begin{align}
\label{eqb51}
0= \sin\lt(3\lt(\theta -\theta'\rt)\rt)\sin\lt(4\lt(\theta -\theta'\rt)\rt)\cos\lt(3\lt(\theta +\theta'\rt)-6 \theta_a\rt)\cos\lt(4\lt(\theta +\theta'\rt)-8 \theta_a\rt)
\end{align}
and
\begin{align}
\label{eqb52}
0=  \sin\lt(\frac{\theta +\theta'}{2}\rt)\cos\lt(\frac{\theta +\theta'}{2}\rt) \sin^2\lt(\frac{\theta -\theta'}{2}\rt)
\end{align}
Since $\theta\neq\theta'$ in $\R/2\pi\Z$, the last factor of \eqref{eqb52} is nonzero, so either the first or the second must be zero. This implies $\theta +\theta'\in\pi\Z$.
As a consequence, the last two factors in \eqref{eqb51} are equal to $\pm\cos(6\theta_a)$ and $\cos(8\theta_a)$ and are nonzero by our choice of $a$. 
So one of the first two factors of \eqref{eqb51} must vanish, that is, 
$\theta -\theta'\in \frac\pi 3\Z\cup \frac\pi 4\Z$.

\medskip

\textit{Step 2.} We have $\theta-\theta'\in \frac{\pi}{4}\mathbb{Z}$. \newline 
\em Proof of Step 2. \rm 
Considering the $M_{13,23}$ minor of $\Gamma_a(\theta)-\Gamma_a(\theta')$ we obtain the equation 
\begin{align*}
0&=\lt(  \sin\lt(8\theta-8\theta_a\rt)  - \sin\lt(8\theta'-8\theta_a\rt) \rt)\nn\\
&\quad\quad\quad\times \lt( \sin\lt(6\rho\lt(\theta \rt)-6\theta_a\rt)  - \sin\lt(6\rho\lt(\theta'\rt) -6\theta_a\rt) \rt) \nn\\
&=4\sin\lt(4\lt(\theta -\theta'\rt)\rt) \cos\lt(4\lt(\theta +\theta'\rt)-8 \theta_a\rt)\nn\\
&\quad\quad \quad \times \sin\lt(3\lt(\rho\lt(\theta \rt)-\rho\lt(\theta'\rt)\rt)\rt)
\cos\lt(3\lt( \rho\lt(\theta \rt)+\rho\lt(\theta'\rt)\rt) -6 \theta_a \rt)
\end{align*}
From Step 1 we have $\theta +\theta'\in\pi\Z$ so the second factor is $\cos(8\theta_a)\neq 0$. The last factor is arbitrarily close to $\pm\cos(6 \theta_a)\neq 0$ since $|\rho-id|\leq\e$. The third factor is nonzero by construction of $\rho$ because $\theta,\theta'\in\frac\pi{24}\Z$ by  \eqref{eqc3}. So we must have $\sin(4(\theta-\theta'))=0$ hence $\theta-\theta'\in\frac\pi 4\Z$.

\medskip

\textit{Step 3.} We have $\theta-\theta'\in \pi\mathbb{Z}$. \newline 
\em Proof of Step 3. \rm Considering the $M_{12,13}$ minor of $\Gamma_a(\theta)-\Gamma_a(\theta')$ we obtain the equation
\begin{align*}
0&=r_a\lt(\cos\lt(\theta\rt)- \cos\lt(\theta'\rt)  \rt)\lt( \sin\lt(8\rho\lt(\theta \rt)-8\theta_a\rt)  - \sin\lt(8\rho\lt(\theta'\rt) -8\theta_a\rt) \rt)   \nonumber\\
&\quad\quad- \lt( \sin\lt(6\theta-6\theta_a\rt)  - \sin\lt(6\theta' -6\theta_a\rt) \rt) \nn\\
&\quad \quad\quad\quad \times \lt( \sin\lt(6\rho\lt(\theta \rt)-6\theta_a\rt)  - \sin\lt(6\rho\lt(\theta'\rt) -6\theta_a\rt) \rt),
\end{align*}
which we rewrite as
\begin{align}
\label{eqc4}
& -4r_a\sin\lt(\frac{\theta +\theta'}{2}\rt)\sin\lt(\frac{\theta-\theta'}{2}\rt)\nn\\
&\quad\quad\quad\times \sin\lt(4\lt(\rho\lt(\theta \rt)-\rho\lt(\theta'\rt)\rt)\rt) \cos\lt(4\lt(\rho\lt(\theta \rt)+\rho\lt(\theta'\rt)\rt)-8 \theta_a\rt)\nn\\
&= 4\sin\lt(3\lt(\theta -\theta'\rt)\rt)\cos\lt(3\lt( \theta +\theta'\rt) -6 \theta_a\rt)\nn\\
&\quad\quad\quad\times\sin\lt(3\lt(\rho\lt(\theta \rt)-\rho\lt(\theta'\rt)\rt)\rt)
\cos\lt(3\lt( \rho\lt(\theta \rt)+\rho\lt(\theta'\rt) -6 \theta_a\rt)\rt).
\end{align}
Since $\theta -\theta'\in\frac\pi 4\Z$ and $|\rho-id|\leq\e$, the third factor in the left-hand side of \eqref{eqc4} has absolute value $\leq 4\e$.
Since $\theta+\theta'\in\pi\Z$, the second factor in the right-hand side of \eqref{eqc4} has absolute value equal to $ | \cos(6\theta_a)|>0$, and for small enough $\e$ the absolute value of the last factor is $\geq |\cos(6\theta_a)|/2>0$.
Taking also into account that the first and third factor in the right-hand side of \eqref{eqc4} differ from  each other by  an error $\leq 3\e$, we must have
\begin{align*}
\sin^2\lt(3\lt(\theta -\theta'\rt)\rt)\leq c_a \e,
\end{align*}
for some $c_a>0$ depending only on $a$.
Because $\theta-\theta'\in\frac\pi 4\Z$ by Step~2, provided $\e$ is chosen small enough, this implies $\theta-\theta'\in \frac\pi 3\Z\cap\frac\pi 4\Z =\pi\Z$.

\textit{Step 4:  Conclusion.}
 From Step~1 and Step~3 we have $\theta+\theta',\theta-\theta'\in\pi\Z$, so $\theta,\theta'\in\pi\Z$. Since me may without loss of generality exchange the roles of $\theta$ and $\theta'$ and choose arbitrary representants in $\R/2\pi\Z$, this amounts to $\theta=0$ and $\theta'=\pi$.
Considering the $M_{12,13}$ minor of $\Gamma_a(0)-\Gamma_a(\pi)$ as in \eqref{eqc4} we deduce
\begin{align*}
\sin(4(\rho(\pi)-\rho(0))\cos(4(\rho(\pi)+\rho(0))+8\theta_a)=0.
\end{align*}
Using again that $|\rho- id|\leq\e$, the second factor has absolute value $\geq |\cos(8\theta_a)|/2$ provided $\e$ is small enough, and the first factor is nonzero by construction of $\rho$, so we conclude that \eqref{eqb1} is not possible for $\theta\neq\theta'\in\R/2\pi\Z$.
\end{proof}

\bibliographystyle{acm}
\bibliography{rigid}

\end{document}